\documentclass[12pt,a4paper]{amsart}
 \pdfoutput=1
 
 \usepackage{amsmath,amssymb,amsthm}
 \usepackage{bm}  
 \usepackage{mathtools}
 \usepackage{thmtools}
 \usepackage{float} 
 
 \usepackage{xcolor} 	
 \usepackage{hyperref}
 \hypersetup{
 	colorlinks,
     linkcolor={red!60!black},
     citecolor={green!60!black},
     urlcolor={blue!60!black},
 }
 \usepackage[maxbibnames=99]{biblatex}
 \usepackage[utf8]{inputenc}
 \usepackage[T1]{fontenc}
 \usepackage{lmodern}
 \usepackage[babel]{microtype}
 \usepackage[english]{babel}
 
 \usepackage{geometry}
 \usepackage{enumitem}

\theoremstyle{plain}
\newtheorem{theorem}{Theorem}[section]
\newtheorem{fact}[theorem]{Fact}
\newtheorem{proposition}[theorem]{Proposition}

\newtheorem{claim}[theorem]{Claim}
\newtheorem{corollary}[theorem]{Corollary}
\newtheorem{lemma}[theorem]{Lemma}

\newtheorem{observation}[theorem]{Observation}

\newtheorem{question}[theorem]{Question}

\newtheorem*{theorem*}{Theorem}
\newtheorem*{corollary*}{Corollary}

\theoremstyle{definition}

\newtheorem{definition}[theorem]{Definition}
\newtheorem{case}{Case}

\title{Weakly reflecting graph properties}

\author{Attila Jo\'{o}}
\thanks{Funded by the Deutsche Forschungsgemeinschaft (DFG, German
Research Foundation)-513023562 and partially by NKFIH OTKA-129211}
\address{Attila Jo\'{o},
Department of Mathematics, University of Hamburg, Bundesstra{\ss}e 55 (Geomatikum), 20146 Hamburg, Germany}
\email{attila.joo@uni-hamburg.de}
\address{Attila Jo\'{o},
Logic, Set theory and topology department, Alfr\'{e}d R\'{e}nyi Institute of Mathematics,  13-15 Re\'{a}ltanoda St., 
Budapest, Hungary}
\email{jooattila@renyi.hu}

\begin{document}

\begin{abstract}
L. Soukup formulated an abstract framework in his introductory paper for proving theorems about uncountable graphs by 
subdividing them by an 
increasing, continuous chain of elementary submodels. The applicability of this method relies on the preservation of a certain property (that varies 
from problem to problem)  by the subgraphs obtained by subdividing the graph by an elementary submodel.  He calls the properties that are 
preserved ``well-reflecting''. The aim of this paper is to investigate the possibility of weakening the assumption 
``well-reflecting'' in L. 
Soukup's framework. Our motivation is to gain a better understanding of a class of problems in infinite graph theory where a weaker form of 
well-reflection naturally occurs.
\end{abstract}
\maketitle
\section{Introduction}
The elementary submodel method is an efficient tool to approach problems in logic, topology and infinitary combinatorics. One way of the
applications is subdividing uncountable structures into smaller well-behaved substructures.  An 
introductory paper about the combinatorial applications of elementary submodels  is written by L. Soukup \cite{soukup2011elementary} in which 
he formulates the 
following 
abstract framework.  A graph\footnote{A graph in this paper is simply a set of unordered pairs.} property (i.e. class of graphs) is well-reflecting if 
whenever  $ M $ is a $ \Sigma 
$-elementary submodel of 
the universe with $ \left|M\right|\subseteq M $ 
for a large enough  finite set $ \Sigma $ of formulas and $ G\in  \Phi \cap M $, then  the subgraphs $ G \cap M
$ and
$ G\setminus M $  are also in $ \Phi $. Let us call these containments the in-reflection and out-reflection of $ \Phi $ 
respectively. Assume that $ 
\Psi\subseteq  \Phi $ are well-reflecting graph classes. Suppose that the countable graphs in $ \Phi $ are also in $ \Psi $, furthermore, if a graph $ G 
$ 
can be partitioned into subgraphs each of which is in $ \Psi $, then $ G \in \Psi $. He shows that then necessarily $ \Phi=\Psi 
$  
\cite[Theorem 
5.6]{soukup2011elementary}. 

In the example given by Soukup, $ \Phi $ 
consists of the graphs not having odd cuts and $ \Psi $ is the class of graphs that can be partitioned into cycles.  He proves directly the 
(highly non-trivial) fact that this $ \Phi $ is indeed well-reflecting  \cite[Lemmas 3.1, 5.2]{soukup2011elementary}. This leads to a new proof 
of the classical theorem of Nash-Williams \cite[p. 235 Theorem 3]{nash1960decomposition} stating that graphs without odd 
cuts are exactly those admitting a 
partition into cycles. 

The main difficulty with the 
application of the 
variants of this framework is showing that the $ \Phi $ corresponding to the problem is well-reflecting (while $ \Psi $ usually 
trivially is). One can typically take  ``one step'' towards the justification of 
property $ \Psi $ without losing property $ \Phi $. For example, if there are no odd cuts, then for every given edge there is a cycle through it and 
after the deletion of the edges of a single cycle, no odd cut will occur. In other 
problems 
the analogous step is much harder. In the 
proof of the infinite version of Menger's theorem\footnote{The original formulation of the framework that we described
accommodates the edge variant of the infinite version of Menger's theorem. } by Aharoni and Berger, they 
link 
one given vertex $ a\in A $ to $ B $ 
by a path $ P $ 
\cite[Theorem 6.1]{aharoni2009menger} in graph $ G $ in such a way that $ G-P $ maintains a rather complicated property $ \Phi $. Here 
$ 
\Psi $ is the 
property that the 
whole $ A $ can be linked to $ B $ by disjoint paths.  A 
similar result has been obtained by the author
concerning the infinite Lovász-Cherkassky problem \cite[Claim 4.6]{joo2023lovcher}. A suitable iteration of such single 
steps  leads  
to the 
solution of 
the corresponding problem in the countable case. It also leads to the solution of the approximation of an uncountable problem by a countable  
elementary submodel, i.e.   $ G 
\cap M\in \Psi $ can be shown if $ G\in \Phi $. This means that the in-reflection of $ \Phi $ restricted to countable 
elementary submodels
can be verified.  

Another 
phenomenon that occurs frequently that one can guarantee that   $ \Phi $ out-reflects most of the time but not 
that always. 
Typically, one considers a carefully chosen 
increasing 
continuous $ \kappa $-chain of subgraphs for a regular uncountable cardinal $ \kappa $.  Then it can be guaranteed that there is a closed 
unbounded set where $ \Phi $ is preserved after the removal of the corresponding initial segment.  See for example 
the so-called 
obstructive $ \kappa $-towers of Aharoni, Nash-Williams and Shelah  \cite{aharoni1983general, aharoni1984konig} or the $ \kappa $-hindrances 
in the already mentioned infinite version of Menger's 
theorem \cite[Section 8]{aharoni2009menger}. 

This paper aims to investigate if Soukup's framework can be improved to handle these arising difficulties. Our motivation is to gain a better 
understanding on an abstract level about some solved problems and obtain a new tool to approach some open ones (for example  \cite[Conjecture 
3.3]{aharoni1998intersection} and \cite[Conjecture 5.1]{joo2023lovcher}). The hindrances explained above led us to introduce 
``weak reflection'' (Definition 
\ref{def: weakly 
reflecting}). We investigate if 
the well-reflection of $ \Phi $ can be replaced by weak reflection in Soukup's framework. We provide the positive answer 
under a rather weak set-theoretic assumption as well as a ZFC-proof for 
graphs of size at most $ \aleph_2 $ (Theorem \ref{thm: main}).
The question if our theorem is provable in ZFC for graphs of arbitrary size remains open.

The paper is 
organized as follows. Our (mostly 
standard) notation and some basic facts are discussed in Section \ref{Sec: notation and basic facts}. The 
main result and its proof are given in Section \ref{Sec: the main result}.

\begin{center}
\textbf{Acknowledgement}
\end{center} 
We are thankful for C. Lambie-Hanson and A. Rinot for bringing the concept of ``approachability'' to our attention. This allowed 
us to replace the weak square principle in our original proof by an even weaker one. 
\\ 

\section{Notation and basic facts}\label{Sec: notation and basic facts}
\subsection{Set theory}The variables $ \alpha,\beta, \gamma $ and $ \delta $ are standing for ordinal numbers, while $ \kappa $ and $ \lambda $ 
denote 
cardinals. The smallest 
limit ordinal, i.e. the set of the natural numbers is denoted by $ \omega $. We write $ 
\mathsf{On} $ for the class of the ordinals, $ \mathsf{acc}(\kappa) $ stands for the accumulation points of $ \kappa 
$.  The cofinality of $ \alpha $ is denoted by $ \mathsf{cf}(\alpha) $ and for cardinals $ \lambda<\kappa $ we let $ 
S^{\kappa}_\lambda:=\{ \alpha<\kappa:\ \mathsf{cf}(\alpha)=\lambda \} $. 
The order type of a set $ O $ of ordinals is $ \mathsf{ot}(O) $. Let $ \kappa $ be 
 an infinite cardinal. A sequence $ \left\langle M_\alpha:\ \alpha<\kappa  \right\rangle $ of sets is increasing ($ \in $-increasing) if $ M_\beta 
 \subseteq 
 M_\alpha $ ($ M_\beta \in
  M_\alpha $) for every $ \beta<\alpha<\kappa $. An increasing sequence is continuous if $ M_\alpha=\bigcup_{\beta<\alpha}M_\beta $ for each $ 
  \alpha\in \mathsf{acc}(\kappa) $. Suppose that $ \kappa=\mathsf{cf}(\kappa)>\aleph_0 $. A set $ C\subseteq \kappa $ is a club of $ \kappa $ if it 
  is unbounded in $ 
 \kappa $ and closed with respect to 
 the order 
 topology (i.e. $ \sup B:=\bigcup B \in C $ for every $ B\subseteq C $ bounded in $ \kappa $). The club filter $ \mathsf{club}(\kappa) $ consists of 
 those subsets 
 of $ \kappa $ that contain a club. The intersection of less than $ \kappa $ many clubs is a club, i.e. this is a $ \kappa $-complete filter. The 
 diagonal intersection $ \triangle_{\alpha<\kappa} C_\alpha $ is defined as $ 
  \bigcap_{\alpha<\kappa} (C_\alpha \cup[0,\alpha]) $ and is a club of $ \kappa $ provided all the $ C_\alpha $ are clubs. A set $ 
 S\subseteq \kappa $ is $ \kappa $-stationary if $ \kappa \setminus S \notin \mathsf{club}(\kappa) $. The set of the 
 stationary  subsets of $ \kappa  $ is denoted by $ \mathsf{stat}(\kappa) $. The concept of approachability was introduced by Shelah 
 implicitly in \cite{shelah1979onsuccessors}.
 \begin{definition}[Approachability ideal, {\cite[first paragraph]{krueger2019theapproach}}]\label{def: approach ideal}
  Let  $ \kappa $ be an uncountable cardinal.  A sequence $ \left\langle a_\alpha:\ \alpha<\kappa^{+}  \right\rangle $ 
   of subsets of $ \kappa^{+} $ with size less than $ \kappa $ is an approaching sequence for $ A\subseteq \kappa^{+} $ if for all  $ \alpha \in A 
   $ there is a $ c_\alpha\subseteq \alpha $ cofinal in $ \alpha $ with $ \mathsf{ot}(c_\alpha)= \mathsf{cf}(\alpha)$ such that $ 
   \gamma \cap c_\alpha \in \{ a_\beta:\ \beta<\alpha \} $ for each $ \gamma<\alpha $. A set $ A\subseteq \kappa^{+} $ is approachable if 
   there 
   exists an 
   approaching sequence for it and $ I[\kappa^{+}] $ is the ideal that consists of those $ A'\subseteq \kappa^{+} $ for which there is an 
   approachable $ A\subseteq A' $ such that $ A'\setminus A $ is non-stationary in $ \kappa^{+} $.
   \end{definition} 
\subsection{Elementary submodels of the universe}Let $ \varphi $ be a first-order formula in the language of set theory with free variables $ 
v_1,\dots, v_n $. For a set $ M $, the formula $ 
\varphi^{M} $ is obtained from $ \varphi $ by  the relativization of the quantifiers to $ M $, i.e. $ \forall v(\dots) $ is replaced by 
$ {\forall v( v\in M \Longrightarrow(\dots))}$ and 
$\exists v(\dots)   $ by $ \exists v( v\in M \wedge (\dots)) $.  The set $ M $ is a
 $ \varphi 
$-elementary submodel of the universe\footnote{It is more common to talk about elementary 
submodels of structures of the form $ H_{\Theta} $ where  $ \Theta $ is a large enough regular cardinal but for our
application this turned out to be less convenient.} if for every $ x_1,\dots, x_n\in M $ we have $  
\varphi(x_1,\dots, 
x_n)\Longleftrightarrow 
\varphi^{M}(x_1,\dots, x_n) $. 
Let $ \Sigma $ be a finite set of formulas. Then we say that $ M $ is a $ \Sigma $-elementary submodel (of the universe) if $ M $ is $ \varphi 
$-elementary for each $ 
\varphi \in \Sigma $. For more details, we refer to \cite{soukup2011elementary}.

\begin{fact}[implicit in {\cite[Corollary 2.6]{soukup2011elementary}}]\label{fact: existence of submodels}
For every set $ x $, infinite cardinal $ \kappa $ and finite set $ \Sigma $ of formulas there exists a $ \Sigma $-elementary  submodel $ M 
$ with $ \left|M\right|=\kappa\subseteq M $ that contains $ x $.
\end{fact}
\begin{fact}[implicit in {\cite[Corollary 2.6]{soukup2011elementary}}]\label{fact: union of chain of submodels}
The union of a $ \subseteq $-chain of $ \Sigma $-elementary submodels is a $ \Sigma $-elementary submodel.
\end{fact}
\begin{fact}[{\cite[Claim 3.7]{soukup2011elementary}}]\label{fact: submodels subset}
There is a finite set $ \Sigma $ of formulas such that if $ M $ is a $ \Sigma $-elementary submodel of size $ \kappa $ with $\kappa\subseteq M $ 
and $ X\in M 
$ with $\left|X\right|\leq \kappa $, then 
$ X\subseteq M $.
\end{fact}
\begin{fact}\label{fact: reg contains size}
If  $ \kappa $ is an infinite cardinal and $ \left\langle M_\alpha:\ \alpha<\kappa  \right\rangle $ is an increasing, continuous 
sequence  with $ \left|M_{\alpha+1}\right|\subseteq M_{\alpha+1} $ and $ M_\alpha \in M_{\alpha+1} $  for 
every $ \alpha<\kappa $, then   $ 
\alpha\cup \left|M_{\alpha}\right|\subseteq M_{\alpha} $ for  $0< \alpha<\kappa $.
\end{fact}
\begin{proof}
Let $ o_\alpha:=\min \mathsf{On}\setminus M_\alpha $. Then $ o_\alpha\subseteq M_\alpha $ by definition and $ \left\langle o_\alpha:\ 
\alpha<\kappa  \right\rangle $ is strictly increasing and continuous, therefore $ \alpha \subseteq o_\alpha $ follows by transfinite induction and 
hence $ \alpha 
\subseteq o_\alpha \subseteq M_\alpha $ for each $ \alpha<\kappa $. 

If $ \alpha\in \mathsf{acc}(\kappa)$, then $ \left|M_\alpha\right|=\sup\{\left|M_{\beta+1}\right|:\ \beta<\alpha \} $ and  $ 
\left|M_{\beta+1}\right|\subseteq M_{\beta+1} \subseteq M_\alpha $ for each $ \beta<\alpha $, thus $ \left|M_\alpha\right|\subseteq M_\alpha $ 
follows. 
\end{proof}

A graph $ G $ in this paper is a set of unordered pairs. 
\section{The main result}\label{Sec: the main result}
\subsection{Preparations}
\begin{definition}[{well-reflecting, \cite[Subsection 5.1]{soukup2011elementary}}]\label{def: well-reflecting}
  A graph class $ \Phi $ is \emph{well-reflecting} if there is a finite set $ \Sigma $ of formulas such that whenever  $ M $ is a $ \Sigma 
  $-elementary 
  submodel of 
  the universe with $ \left|M\right|\subseteq M $, then  
  for every $ G\in  \Phi \cap M $ we have  $ G \cap M, G\setminus M \in \Phi $.
\end{definition}
\begin{definition}[weakly reflecting]\label{def: weakly reflecting}
A graph class $ \Phi $ is  \emph{weakly reflecting} if there is a finite set $ \Sigma $ of formulas such that
\begin{itemize}
\item whenever  $ M $ is a countable $ \Sigma $-elementary submodel of the universe and $ G\in  \Phi \cap M $,  then  $ G 
\cap M\in \Phi  $; 
 \item if  $ \kappa $ is an uncountable regular cardinal and $ \left\langle M_\alpha: \alpha<\kappa  
 \right\rangle 
 $ is an 
 increasing, continuous and $ \in $-increasing sequence of  $ \Sigma $-elementary submodels of the universe such that
  $ \kappa>\left|M_\alpha\right|\subseteq M_\alpha $  for every $ \alpha<\kappa $, then for each $ G \in  
  \Phi \cap 
  M_0 $ there is a club $ 
  C_G $ of $ 
  \kappa $ such that  $ G \setminus M_\alpha\in \Phi $  for every $ \alpha \in C_G $.
\end{itemize} 
\end{definition}

\begin{theorem}\label{thm: main}
Suppose that  $  \Psi\subseteq \Phi $ are classes of graphs such that: 
\begin{enumerate}[label=(\Roman*)]
\item\label{item: main countable}The countable graphs in $ \Phi $ are also in $ \Psi $;
\item\label{item: main merging}The class $ \Psi $ is closed under taking the union of arbitrary many pairwise disjoint elements.
\item\label{item: main club subtractable}The property $ \Phi $ is  weakly reflecting and $ \Psi $ is well-reflecting.
\end{enumerate}
Then every $ G \in \Phi $ with $ \left|G\right|\leq \aleph_2 $ is in $ \Psi $. Furthermore, if for every uncountable 
regular $ \kappa $ the 
set $  S^{\kappa^{+}} _{\kappa}$ has a 
stationary subset 
in $ I[\kappa^{+}] $, then $ \Phi=\Psi $.
\end{theorem}

\begin{proof}
Let $ \Sigma $ be a fixed large enough finite set of formulas that accommodates the weak reflection of $ \Phi $, the well-reflection of $ \Psi $ and 
contains some finitely many additional formulas which will be implicitly defined in the proof. To improve the flow of words we will write simply 
`elementary submodel' instead of `$ \Sigma $-elementary submodel of the universe'. First, we prove the ``Furthermore'' part of 
Theorem \ref{thm: main} and point out explicitly which step requires the approachability assumption.  

For a set $ M $ 
and  graph $ G $ we say that $ M $ is $ G 
$\emph{-intersectable} if $ 
G\cap 
M\in \Psi $. Furthermore, $ 
M $ is $G $\emph{-subtractable} if $G\setminus M \in \Phi$. We call 
$ M $ \emph{intersectable}  (\emph{subtractable})   if
$ M $  is $G $-intersectable ($G $-subtractable) for  every $G\in \Phi\cap M $.

\begin{observation}\label{obs: countable intersectable}
Countable elementary submodels are intersectable.
\end{observation}
\begin{proof}
Let $ M $ be a countable elementary submodel and $ G \in \Phi \cap M $. By weak reflection, we have $ G \cap M \in \Phi $, thus 
by property \ref{item: main countable} $ G \cap M \in \Psi $.
\end{proof}

\begin{definition}\label{def: decomp}
 We say that $ \mathcal{D} $ is a \emph{decomposition} of the elementary submodel $ M $ of size $ \kappa $ if one of the 
 following conditions hold:
 
 \begin{enumerate}
 \item $ \kappa=\aleph_0 $ and $ \mathcal{D}=M $;
 \item\label{item: def decomb reg} $ \kappa $ is an uncountable regular cardinal and $ \mathcal{D}=\left\langle M_\alpha:\ 
 \alpha<\kappa  \right\rangle  $ is an 
 increasing, continuous sequence  of intersectable and subtractable  elementary submodels with $ 
 \bigcup_{\alpha<\kappa}M_\alpha=M $ such that for every $ \alpha<\kappa $:
 \begin{enumerate}
  \item\label{item: reg size} $ \left|M_\alpha\right|<\kappa $,
 \item\label{item: reg contains size} $ \left|M_{\alpha}\right|\subseteq M_{\alpha} $,
 \item\label{item: reg contains previous} $ M_\alpha \in M_{\alpha+1} $;
 \end{enumerate}
 \item\label{item: def decomb sing} $ \kappa $ is a singular cardinal, $ \left\langle \kappa_\alpha:\ \alpha<\mathsf{cf}(\kappa)  \right\rangle $ is an 
 increasing, continuous 
  sequence
    of 
  cardinals with limit $ \kappa $ in which $\kappa_0>\mathsf{cf}(\kappa) $ and  
  \[ \mathcal{D}=\{ M_{\alpha, n}:\ \alpha<\mathsf{cf}(\kappa), n<\omega \} \]
 is a family of intersectable and subtractable elementary submodels with $ \bigcup_{\alpha<\mathsf{cf}(\kappa), 
 n<\omega}M_{\alpha, n}=M $ such that for every $ \alpha<\mathsf{cf}(\kappa) $ and $ n<\omega $:
 \begin{enumerate}
  \item\label{item: sing size} $ \left|M_{\alpha,n}\right|=\kappa_\alpha $,
  \item\label{item: sing contains size} $ \left|M_{\alpha,n}\right|\subseteq M_{\alpha,n} $,
  \item\label{item: sing contains previous} $ M_{\alpha', n'}\in M_{\alpha,n}$ if either  $ n'<n $ or $ n'=n $ and $ 
 \alpha'<\alpha $.
 \end{enumerate}
 \end{enumerate}
  We call $ M $ \emph{decomposable} if it has a decomposition.
\end{definition}
\begin{observation}\label{obs: decomp contains size}
If $ M $ is decomposable then $ \left|M\right|\subseteq M $. 
\end{observation}
\begin{proof}
It is obvious if $ M $ is countable. If $ \kappa>\aleph_0 $ is regular and $ \left\langle M_\alpha:\ \alpha<\kappa  \right\rangle $ is a decomposition 
of $ M $, then $ \alpha \subseteq M_\alpha $ by  
Fact \ref{fact: reg contains size} and therefore $ \kappa \subseteq M $.  Finally, if $ \kappa:=\left|M\right| $ is singular and $ \{ M_{\alpha, 
n}:\ \alpha<\mathsf{cf}(\kappa), n<\omega \} $ is a decomposition 
of $ M $, then $ 
\kappa_\alpha \subseteq 
M_{\alpha,0}\subseteq M 
$ for every $ \alpha<\mathsf{cf}(\kappa) $ by assumption, thus by taking union for $ \alpha<\mathsf{cf}(\kappa) $ we conclude $ \kappa 
\subseteq M $. 
\end{proof}

We are going to show that decomposable elementary 
submodels are intersectable and under the approachability assumption in Theorem \ref{thm: main} they can be constructed 
of any size containing a 
prescribed set $ x $. This implies $ \Phi=\Psi $. Indeed,  for a given $ G\in\Phi $ we construct a decomposable elementary 
submodel $ M $ of size $ \kappa:=\left|G\right| $ with $ G \in M $. Then $ \left|M\right|\subseteq M $ by Observation 
\ref{obs: decomp contains size}. Thus $G \subseteq M $ by Fact \ref{fact: submodels subset} and hence 
$  G \cap M=G $. Since $ M $ is intersectable, we obtain  $ G=G\cap M \in \Psi $.

\begin{observation}\label{obs: chain very basic} 
Assume that $\kappa $ is an infinite cardinal, $G\in \Phi$ and  $ \left\langle M_\alpha:\ \alpha<\kappa  
\right\rangle $  
is an increasing, continuous 
sequence 
of elementary 
submodels such that $ M_0 $ is $G $-intersectable  and $ M_{\alpha+1} $ is $ G\setminus M_{\alpha} 
$-intersectable for every $ \alpha<\kappa $. Then $ 
M:=\bigcup_{\alpha<\kappa}M_{\alpha} $ is a $G $-intersectable elementary submodel.
\end{observation}
\begin{proof}
By Fact \ref{fact: union of chain of submodels}, $ M $ is an elementary submodel. We have $ G\cap M_0\in \Psi $ 
because $ M_0 $ is $ G $-intersectable. Similarly $ (G\setminus M_\alpha) \cap 
M_{\alpha+1}\in \Psi $ because $ M_{\alpha+1} $ is $ G\setminus M_\alpha $-intersectable. Since  $ G\cap M_0$ together with 
$ (G\setminus M_\alpha) \cap M_{\alpha+1} \ (\alpha<\kappa) $ forms an partition of $ G\cap M 
$,  we obtain $ G\cap M\in \Psi $ by  property \ref{item: main merging}.
\end{proof}

\begin{corollary}\label{cor: chain basic}
Assume that $ \kappa $ is an infinite cardinal and    $ \left\langle M_\alpha:\ \alpha<\kappa  
\right\rangle 
$ is an increasing, continuous 
sequence of  subtractable elementary 
submodels such that $ M_\alpha \in 
M_{\alpha+1} $  and $ M_{\alpha+1 }$ is intersectable for each $ \alpha<\kappa $. Then $ 
M:=\bigcup_{\alpha<\kappa}M_{\alpha} $ is an intersectable elementary submodel.
\end{corollary}
\begin{proof}
Let $G\in \Phi\cap M $ be given. We may assume without loss of generality that $G\in M_0 $ and $ M_0 $ is 
also intersectable since otherwise we switch to a suitable terminal segment of the sequence. To reduce the 
statement 
to Observation \ref{obs: chain very basic}, it is enough to show that $ M_{\alpha+1} $ is $G\setminus M_\alpha $-intersectable for every $ 
\alpha<\kappa $. Since $G, M_\alpha \in M_{\alpha+1}  $, we have $ 
G \setminus M_\alpha\in M_{\alpha+1} $. Furthermore, $G \setminus M_\alpha \in \Phi $  because  $ M_\alpha $ is subtractable. But then the 
intersectability of $ 
M_{\alpha+1} $ implies that it is,, in particular,, $ G \setminus M_\alpha$-intersectable.
\end{proof}

\begin{corollary}\label{cor: decomp reg intersectable}
If $ M $ is a decomposable elementary submodel and $ \kappa:=\left|M\right| $ is a regular cardinal, then $ 
M $ is intersectable.
\end{corollary}
\begin{proof}
For $\kappa=\aleph_0 $,  it follows from Observation \ref{obs: countable intersectable}. If $ \kappa $ is an uncountable regular 
cardinal, then we apply Corollary \ref{cor: chain basic} to a decomposition of $ M $ to conclude that $ M $ is intersectable.
\end{proof}

The following claim is an application of the singular compactness method by Shelah.
\begin{claim}\label{claim: decomps sin intersectable}
 If $ M $ is a decomposable elementary submodel and $ \kappa:=\left|M\right| $ is singular, then $ 
 M $ is intersectable.
\end{claim}

\begin{proof}
Let $ \{M_{\alpha,n}:\   \alpha<\mathsf{cf}(\kappa), n < \omega \} $ be a decomposition of $ M $. We 
have $ 
M_{\beta,n}\subseteq 
M_{\alpha,n} $ and $ M_{\beta,n}\subseteq M_{\beta, n+1} $ 
whenever $ 
\beta<\alpha<\mathsf{cf}(\kappa) $ and $ 
n<\omega $ by Fact \ref{fact: submodels subset} via the properties (\ref{item: sing size}), (\ref{item: sing contains size})  and 
(\ref{item: sing contains previous}) of Definition \ref{def: decomp}.
Thus the set $ M_\alpha:= \bigcup_{n<\omega}M_{\alpha,n} $ is  an elementary submodel by 
Fact \ref{fact: union of chain of submodels}.  By considering the inclusion $ M_{\beta,n}\subseteq 
M_{\alpha,n} $ and taking union for $ n<\omega $,   we conclude that $ 
M_\beta \subseteq M_\alpha $ for $ \beta<\alpha<\mathsf{cf}(\kappa) $. 
Thus  Fact \ref{fact: union of chain of submodels} ensures that the set  $ 
M=\bigcup_{\alpha<\mathsf{cf}(\kappa)}M_{\alpha} $ is also  an elementary 
submodel.
\begin{lemma}
$ M_\alpha=\bigcup_{\beta<\alpha}M_\beta $ for every  $ \alpha\in \mathsf{acc}(\mathsf{cf}(\kappa)) $.
\end{lemma}
\begin{proof}
The inclusion $ M_\alpha\supseteq \bigcup_{\beta<\alpha}M_\beta $ is obvious since we have already seen that $ 
M_\beta \subseteq M_\alpha $ for $ \beta<\alpha<\mathsf{cf}(\kappa) $. It remains to prove that $ 
M_\alpha\subseteq \bigcup_{\beta<\alpha}M_\beta $. To do so, it is enough to show that $ M_{\alpha,n}\subseteq \bigcup_{\beta<\alpha}M_\beta 
$ for each $ n<\omega $. Let $ n<\omega $ be fixed. Then $ M_{\alpha,n}\in M_{0,n+1}\subseteq M_0\subseteq \bigcup_{\beta<\alpha}M_\beta 
$
by property  (\ref{item: sing contains previous}). By property (\ref{item: sing size}), we have $ \left|M_{\beta,n} \right|=\kappa_\beta $ and hence 
$ \left|M_\beta\right|=\kappa_\beta $ for every $ \beta<\mathsf{cf}(\kappa) $ 
and  $ \kappa_\alpha=\bigcup_{\beta<\alpha}\kappa_\beta $ since $ \left\langle \kappa_\beta:\ \beta<\mathsf{cf}(\kappa)  \right\rangle $ is 
assumed to 
be continuous. Thus 
\[ \left|M_{\alpha,n}\right|=\kappa_\alpha=\sum_{\beta<\alpha}\kappa_\beta=\sum_{\beta<\alpha}\left|M_\beta\right| 
=\left|\bigcup_{\beta<\alpha}M_\beta \right|. \] Finally, Fact 
\ref{fact: submodels subset} ensures $ M_{\alpha,n}\subseteq \bigcup_{\beta<\alpha}M_\beta$.
\end{proof}
To show that $ M $ is intersectable we fix a $G\in \Phi\cap M $. We may assume without loss of 
generality that $G\in M_{0,0} $, since otherwise we switch to another decomposition of $ M $ by deleting an initial segment of the rows 
and columns of the original decomposition.  The sequence $ \left\langle M_{0,n}:\ n<\omega  \right\rangle $ satisfies the premise of 
Corollary \ref{cor: chain basic}, thus $ 
M_0 $ is intersectable and in particular $ 
G $-intersectable. We intend to apply Observation \ref{obs: chain very basic}, thus it remains to show that
\begin{lemma}
For every $ \alpha<\mathsf{cf}(\kappa) $, $ M_{\alpha+1} $ is $ G\setminus M_{\alpha} $-intersectable.
\end{lemma}
\begin{proof}
The graphs $ G\cap M_{\alpha+1,0} $ and  $ (G\setminus M_{\alpha+1,n}) \cap M_{\alpha+1,n+1}\ 
(n<\omega) $ are elements of $ M_\alpha $ (see property (\ref{item: sing contains previous})) and form an partition of $ 
G\cap M_{\alpha+1} $. It is enough to show that  they are in $ \Psi $. Indeed, then $ (G\cap 
M_{\alpha+1,0})\setminus M_\alpha $ and  $ \left[(G\setminus M_{\alpha+1,n}) \cap 
M_{\alpha+1,n+1} \right]\setminus M_\alpha \ 
(n<\omega) $ are also in $ \Psi $ because $ \Psi $ is well-reflecting by \ref{item: main club subtractable}. Since they form a 
partition of 
$  (G\cap M_{\alpha+1}) \setminus M_\alpha =(G\setminus M_\alpha) \cap M_{\alpha+1} $, we will 
be done by property \ref{item: main merging}. 

Clearly $ G\cap M_{\alpha+1,0} \in \Psi$ because $ M_{\alpha+1,0} $ is intersectable and $ G \in \Phi \cap M_{\alpha+1,0} $.
Since $ M_{\alpha+1,n} $ is subtractable and contains $ G $, we have $ G\setminus  M_{\alpha+1,n} \in \Phi $. Then  
$ (G\setminus  M_{\alpha+1,n})\cap M_{\alpha+1,n+1} \in \Psi$ because $ M_{\alpha+1,n+1} $ is intersectable and contains  $ 
G\setminus  M_{\alpha+1,n} $. This concludes the proof.
\end{proof}
It follows from Observation \ref{obs: chain very basic} that $ M $ is $G $-intersectable. Since $G\in \Phi\cap M $ 
was arbitrary it means that $ M $ is intersectable.
\end{proof}

\begin{proposition}\label{prop: decomp is intersectable}
Every decomposable elementary submodel $ M $ is intersectable.
\end{proposition}
\begin{proof}
It is immediate from Corollary \ref{cor: decomp reg intersectable} and Claim \ref{claim: decomps sin intersectable}.
\end{proof}

\subsection{Finding a club of intersectable and subtractable elementary submodels}
In the following lemmas, we have a $ \kappa $-chain of elementary submodels  and we are looking for a club of $ \kappa 
$ such that the corresponding submodels satisfy certain properties.

\begin{lemma}\label{lem: subtractable club exist} 
For every uncountable regular cardinal $ \kappa $ and every  increasing, continuous sequence $ \left\langle M_\alpha:\ 
\alpha<\kappa  
\right\rangle $ of elementary submodels satisfying (\ref{item: reg size}), (\ref{item: reg contains size})  and 
(\ref{item: reg contains previous}) of Definition \ref{def: decomp},
\[ \{ \alpha<\kappa:\ M_\alpha \text{ is subtractable} \}\in \mathsf{club}(\kappa). \]
\end{lemma}
\begin{proof}
Let $ M:= \bigcup_{\alpha<\kappa} M_\alpha$. For every  $ 
G \in  \Phi\cap M $, let $ C_{G} $ be a club of $ \kappa $ such that  $ 
M_\alpha $ is $G $-subtractable for  $ \alpha \in  C_{G} $ (see Definition \ref{def: weakly reflecting}). For $ \alpha<\kappa $ let 
\[ C_\alpha:=\bigcap\{ C_G:\ G\in \Phi \cap M_\alpha  \}. \]

 Note that $ C_\alpha $ is a club because  $ 
\left|M_\alpha\right|<\kappa 
$ ensures that $ C_\alpha $ is the intersection of less than $ \kappa $  clubs. Let
\[ C:=\mathsf{acc}(\kappa)\cap\triangle_{\alpha<\kappa} C_\alpha. \]
Then $ C $ is a club.  Let $\alpha \in C $ be fixed  and take a $G\in \Phi\cap M_\alpha $. Since $ \alpha $ is a 
limit ordinal and the sequence $ \left\langle M_\alpha:\ \alpha<\kappa  
\right\rangle $ is continuous there is a $\beta<\alpha $ with $G\in M_{\beta} $. By the definition of $ C $ 
we know that $ \alpha \in 
C_\beta \subseteq C_{G} $. But then, by the definition of $ C_{G} $ we conclude that $ M_\alpha $ 
is $G $-subtractable. Since $ G \in M_\alpha \cap \Phi $ was arbitrary, it means that $ M_\alpha $ is subtractable.

\end{proof}
\begin{lemma}\label{lem: str many intersectable}
Let $ \kappa $ be an uncountable regular cardinal  and let 
$ \left\langle M_\alpha:\ \alpha<\kappa  \right\rangle $ be an increasing, continuous sequence  of elementary submodels 
satisfying (\ref{item: reg size}), (\ref{item: reg 
contains size})  and 
(\ref{item: reg contains previous}) of Definition \ref{def: decomp}. If  $I:= \{ \alpha<\kappa:\ M_\alpha\text{ is 
intersectable}\}\in \mathsf{stat}(\kappa) $, then

\[ \{ \alpha<\kappa:\ M_\alpha \text{ is subtractable and intersectable} \}\in \mathsf{club}(\kappa). \]
\end{lemma}
\begin{proof}
By applying Lemma \ref{lem: subtractable club exist}, we pick a club $ C $ of $ \kappa $ such that $ M_\alpha $ is subtractable for each $ 
\alpha\in C $.  Then $ I\cap C $ is unbounded in $ \kappa $ and $ 
M_\alpha $ is intersectable and subtractable for $ \alpha \in I\cap C  $ by definition. We claim that $ I\cap C $ is closed as 
well. Suppose for a contradiction that it is not and let $\beta = \min \overline{I\cap C}\setminus I\cap C $. Let $ 
\left\langle \beta_{\xi}: \xi< \mathsf{cf}(\beta)  \right\rangle $ be an increasing, continuous sequence with limit $ \beta $. 
Then  $ 
\left\langle M_{\beta_\xi}:\ 
\xi<\mathsf{cf}(\beta)   \right\rangle $ is an increasing, continuous sequence of intersectable and subtractable elementary 
submodels with 
$\bigcup_{\xi<\mathsf{cf}(\beta) }M_{\beta_\xi} = M_{\beta}$  and 
$ M_{\beta_\xi}\in M_{\beta_{\xi+1}} $ for each $ \xi<\mathsf{cf}(\beta) $, thus $ \beta \in I $ by Corollary \ref{cor: chain 
basic}.  
But then  $ \beta \in I \cap C $ because $ C $ is closed and this 
contradicts the choice of $ \beta $. Therefore $I\cap C $ is closed and thus it is a 
desired club.
\end{proof}

\begin{lemma}\label{lem: succ intersectable and subtractable}
Let $ \kappa $ be an uncountable regular cardinal  and let 
$ \left\langle M_\alpha:\ \alpha<\kappa  \right\rangle $ be an increasing, continuous sequence  of  elementary submodels 
satisfying (\ref{item: reg size}), (\ref{item: reg 
contains size})  and 
(\ref{item: reg contains previous}) of Definition \ref{def: decomp}. If  $ M_{\alpha+1} $ is subtractable and 
intersectable for every $ \alpha<\kappa $, then

\[ \{ \alpha<\kappa:\ M_\alpha \text{ is subtractable and intersectable} \}\in \mathsf{club}(\kappa). \]
\end{lemma}
\begin{proof}
It is enough to show that $I:= \{ \alpha<\kappa:\ M_\alpha\text{ is intersectable}\}\in \mathsf{stat}(\kappa) $ because then Lemma \ref{lem: 
str many intersectable} provides the desired club. To do so, it is sufficient to prove that  $I \supseteq S^{\kappa}_{\omega}  
$. Let $ \alpha \in 
S^{\kappa}_{\omega} $ be given. We need to show that  $ 
M_\alpha $ is intersectable.  Take an increasing  sequence  $ \left\langle \beta_n:\ n<\omega  \right\rangle 
$ of successor ordinals with limit $ \alpha 
$. Then $ \left\langle M_{\beta_n}:\ n<\omega  \right\rangle $  is an increasing (and continuous) 
sequence of intersectable and 
subtractable elementary submodels  with $ \bigcup_{n<\omega }  M_{\beta_n}= M_\alpha$  
and $ M_{\beta_n} \in M_{\beta_{n+1}} $ for $ n<\omega $, thus $ 
M_\alpha $ is intersectable by Corollary \ref{cor: chain basic}. 
\end{proof}

\subsection{The existence of decomposable elementary submodels}
It remains to construct decomposable submodels of a given size $ \kappa $ containing a prescribed set $ x $. For technical reasons we make the 
second requirement stronger and demand the existence of a decomposition in which every member contains $ x $:
\begin{proposition}\label{prop: decomp exist}
For every set $ x $ and infinite cardinal $ \kappa $, there exists an elementary submodel $ M $ of size $ \kappa $  that admits 
a decomposition $ \mathcal{D} $ in which each elementary submodel contains $ x $. Furthermore, if $ \kappa $ is singular and a 
sequence $ \left\langle \kappa_\alpha:\ \alpha<\mathsf{cf}(\kappa)  \right\rangle $  as in Definition \ref{def: decomp} is given, then $ \mathcal{D} 
$ can be chosen accordingly.
\end{proposition}
\begin{proof}
Let $ x $ be fixed. We apply transfinite induction on $ \kappa $. 
\begin{case}\label{case: alep0}
$ \kappa=\aleph_0 $
\end{case}
Any elementary submodel containing $ x $ is 
suitable (see Definition \ref{def: decomp}).\\
 
Suppose now that $ \kappa>\aleph_0 $. 
\begin{lemma}\label{lem: subtractable intersectable smaller enough}
If for every set $ x' $ and every infinite cardinal $ \lambda<\kappa $ there is a intersectable and subtractable elementary submodel $ M' 
$  with $ \lambda=\left|M'\right|\subseteq M' $ that 
contains $ x' $, then Proposition \ref{prop: decomp exist} holds for $ \kappa $.
\end{lemma}
\begin{proof}
Roughly speaking, we build a decomposition by transfinite recursion based on the assumption and define $ M $ accordingly. Assume first that $ 
\kappa $ is a regular. We are going to build an increasing, continuous sequence $ \left\langle M_\alpha:\ 
\alpha<\kappa  \right\rangle $ of elementary submodels containing $ x $  satisfying (\ref{item: reg size}), 
(\ref{item: reg contains size})  and 
(\ref{item: reg contains previous}) in which $ M_{\alpha+1} $ 
is subtractable and intersectable 
for every $ \alpha<\kappa $ and set $ M:=\bigcup_{\alpha<\kappa}M_\alpha $. After this is done,  Lemma 
\ref{lem: succ intersectable and subtractable} provides a club of $ \kappa $ such that the 
corresponding subsequence is a decomposition of $ M $. Let $ M_0 $ be any countable elementary submodel that contains $ x' $. If  $ 
\alpha\in \mathsf{acc}(\kappa) $ and $ M_\beta $ is defined for $ \beta<\alpha $, then we take $ 
M_\alpha:=\bigcup_{\beta<\alpha}M_\beta $.  If $ \alpha<\kappa $ and $ M_\alpha $ is already defined,  then by assumption  we can take a  
intersectable 
and subtractable elementary submodel $ M_{\alpha+1} $ with 
$\left|\alpha\right|+\aleph_0=\left| 
M_{\alpha+1} 
 \right|\subseteq M_{\alpha+1} $  that contains $M_\alpha$. It follows directly from the construction that $ \left\langle 
 M_\alpha:\ 
 \alpha<\kappa  \right\rangle $ satisfies the premise of Lemma \ref{lem: succ intersectable and subtractable}.

Assume now that $ \kappa $ is a singular cardinal. Let $ \left\langle \kappa_\alpha:\ \alpha<\mathsf{cf}(\kappa)  \right\rangle $ be a given 
increasing, 
continuous sequence of cardinals with limit $ \kappa $ and $\kappa_0>\mathsf{cf}(\kappa) $. We let $ M_{0,0} $ to be an elementary submodel 
with $ \kappa_{0}=\left|M_{0,0}\right|\subseteq M_{0,0} $ that contains $ x $.  Suppose that there is some $ 
\alpha<\mathsf{cf}(\kappa) $ and $ n<\omega $ such that $ (\alpha,n)\neq (0,0) $ and $ M_{\alpha', n'} $ is already defined whenever    $ n'<n $ or 
$ n'=n $ and $ 
 \alpha'<\alpha $. Then we pick a  intersectable and subtractable elementary submodel $ M_{\alpha,n} $ with $\kappa_\alpha=\left| 
 M_{\alpha,n} 
 \right|\subseteq M_{\alpha,n} $  that contains $\{ M_{\alpha', n'}:\ (n'<n)\vee 
 (n'=n)\wedge 
 (\alpha'<\alpha) \}$.  
Then
  $M:=\bigcup_{\alpha<\mathsf{cf}(\kappa), n<\omega}M_{\alpha,n} $ is  a desired decomposable elementary 
  submodel.
\end{proof}

\begin{case}\label{case: limit cardinal}
$ \kappa $ is a limit cardinal.
\end{case}

\begin{observation}\label{obs: pick intersectable and subtractable}
For every set $ x' $ and infinite cardinal $ \lambda <\kappa $, there is a intersectable and subtractable elementary submodel $ M' 
$  with $ \lambda=\left|M'\right|\subseteq M' $ that 
contains $ x' $.
\end{observation}
\begin{proof}
By applying the induction hypotheses of Proposition \ref{prop: decomp exist} with $ x' $ and $ \lambda^{+}<\kappa $, we obtain a decomposable 
$ 
M $ of size $ \lambda^{+} $ and a decomposition $ \left\langle M_\alpha:\ \alpha<\lambda^{+}  \right\rangle $ of it where 
each $ M_\alpha 
$ contains $ x' $. For every large enough $ \alpha<\lambda^{+} $ we must have $ 
\left|M_\alpha\right| =\lambda $. Let  $M':= M_{\alpha} $ for the smallest  $ \alpha $ for which $ \left|M_\alpha\right| =\lambda $.
\end{proof}

Observation \ref{obs: pick intersectable and subtractable} ensures that the premise of Lemma \ref{lem: subtractable intersectable smaller enough} 
holds. This concludes the 
induction step when $ \kappa $ is a limit cardinal.\\

It remains to prove Proposition \ref{prop: decomp exist} for successor cardinals.  
\begin{case}\label{case}
$ \kappa=\omega_1 $
\end{case}
We take an increasing, continuous sequence $ \left\langle M_\alpha:\ \alpha<\omega_1  \right\rangle $ of 
countable elementary submodels with $ x \in M_0 $ and $ M_\alpha \in M_{\alpha+1} $ for $ \alpha<\omega_1 $ 
and set $ M:=\bigcup_{\alpha<\omega_1} M_\alpha $.  Let $ C $ be a club of $ \omega_1 $ such that $ M_\alpha $ 
is subtractable for every $ \alpha \in C $ (exists by Lemma \ref{lem: subtractable club exist}). Since countable elementary submodels are 
intersectable (see Observation 
\ref{obs: countable intersectable}), the subsequence of $ 
\left\langle M_\alpha:\ \alpha<\omega_1  \right\rangle $ corresponding to $ C $ is a decomposition of $ M $.\\

 Suppose now that $ \kappa=\lambda^{+} $ where $ \lambda>\aleph_0 $ and fix an increasing, 
continuous sequence $ \left\langle \lambda_\alpha:\ \alpha<\mathsf{cf}(\lambda ) \right\rangle $ of cardinals with  $ 
\lambda_0>\mathsf{cf}(\lambda) $ and 
limit $ 
\lambda $, if $ \lambda $ is singular.  We are going to build an increasing, continuous 
sequence $ \left\langle M_\alpha:\ \alpha<\lambda^{+}  \right\rangle $ of elementary submodels  with $ x \in 
M_0 $ 
together with a sequence $ \left\langle 
 \mathcal{D}_\alpha:\ \alpha<\lambda^{+}  
 \right\rangle 
 $  such that for every $ \alpha<\lambda^{+} $:
 
 \begin{enumerate}[label=(\roman*)]
 \item $ \left|M_\alpha\right|=\lambda $,
 \item $ \left|M_\alpha\right| \subseteq M_\alpha $,
 \item\label{item: superdecom reg} $ \mathcal{D}_{\alpha+1}=\left\langle M^{\alpha+1}_{\beta}:\ \beta<  \lambda \right\rangle $ is a 
 decomposition  of $ 
 M_{\alpha+1} $ with $ \left\langle (M_\beta, \mathcal{D}_\beta):\   \beta \leq \alpha\right\rangle \in 
 M^{\alpha+1}_0 $  if $ 
 \lambda $ is regular,
 \item\label{item: superdecom sing} $ \mathcal{D}_{\alpha+1}= \{ M^{\alpha+1}_{\beta,n}:\ \beta<  \mathsf{cf}(\lambda),\ n<\omega 
 \} 
 $   is a decomposition  of $ 
  M_{\alpha+1} $ with $ \left\langle (M_\beta, \mathcal{D}_\beta):\   \beta \leq \alpha\right\rangle \in 
  M^{\alpha+1}_{0,0} $ 
  corresponding to $ \left\langle \lambda_\alpha:\ \alpha<\mathsf{cf}(\lambda ) \right\rangle $ if $ \lambda $ is singular.
 \end{enumerate} 
 We let $ \mathcal{D}_\alpha=\emptyset $ for $ \alpha \in \{ 0 \}\cup \mathsf{acc}(\lambda^{+}) $. 

 \begin{observation}
 $ \left\langle M_\alpha:\ \alpha<\lambda^{+}  
 \right\rangle $ satisfies (\ref{item: reg size}), (\ref{item: reg contains size}) and (\ref{item: reg contains previous}) of Definition 
 \ref{def: decomp}.
 \end{observation}
 \begin{observation}\label{cor: superdecomp contains}
 $ M_{\beta}\cup \mathcal{D}_\beta \cup \{M_{\beta}, \mathcal{D}_\beta \} \subseteq M_{\alpha} $  for $ \beta<\alpha $.
 \end{observation}
 \begin{proof}
Since $ \alpha \subseteq M_\alpha $ by Fact \ref{fact: reg contains size}, it follows from properties  \ref{item: superdecom reg} and  
\ref{item: superdecom sing} and Fact \ref{fact: submodels subset}.
 \end{proof}

The construction of $ \left\langle M_\alpha:\ \alpha<\lambda^{+}  \right\rangle $ and $ \left\langle 
 \mathcal{D}_\alpha:\ \alpha<\lambda^{+}   \right\rangle $  can be done by a straightforward transfinite recursion. We apply 
 the induction hypotheses 
with respect to Proposition \ref{prop: decomp exist} with the set $ x$ and cardinal $ \lambda $ to get $ 
M_0 $ (see also Observation 
\ref{obs: decomp contains size}), take  
$ M_\alpha:=\bigcup_{\beta<\alpha}M_\beta $ if $ \alpha\in \mathsf{acc}(\lambda^{+})  $ 
and  use the induction hypotheses 
 with set $ \left\langle (M_\beta, \mathcal{D}_\beta):\   \beta \leq \alpha\right\rangle $ and cardinal $ \lambda $ to 
 get $ M_{\alpha+1} $ and $ \mathcal{D}_{\alpha+1} $. The recursion is done and we set $ 
M:=\bigcup_{\alpha<\lambda^{+}}M_\alpha $. 

We are going to show 
that there is a club of $ \lambda^{+} $ such that the corresponding subsequence of $ \left\langle 
M_\alpha:\ \alpha<\lambda^{+} \right\rangle $ is a decomposition of $ M $, in other words  

\[ \{ \alpha<\lambda^{+}:\ M_\alpha \text{ is subtractable and intersectable} \}\in \mathsf{club}(\lambda^{+}). \]

By  Lemma \ref{lem: str many intersectable} it is enough to show that
\[ I:=\{ \alpha<\lambda^{+}:\ M_\alpha \text{ is intersectable} \}\in \mathsf{stat}(\lambda^{+}). \]

\noindent Roughly speaking, the idea is to construct a decomposition of $ M_\alpha $ for stationarily many $ \alpha \in I $  ``diagonally'' by  using 
the 
members of the 
decompositions of 
some previous terms. We have two cases depending on if $ \lambda $ is regular or singular. 
\begin{case}\label{case: lambda singular}
$ \kappa=\lambda^{+} $ with $\lambda> \mathsf{cf}(\lambda) $
\end{case}
It is sufficient to show that $I \supseteq S^{\lambda^{+}}_{\omega} $. Let $ \alpha \in 
S^{\lambda^{+}}_{\omega} $ be fixed and take an increasing  sequence $ \left\langle \beta_n:\ n<\omega 
\right\rangle $ of successor ordinals with limit $ \alpha $. We pick recursively $ \gamma_n<\mathsf{cf}(\lambda) $ and $ i_n<\omega $ and set $ 
M'_n:=M^{\beta_n}_{\gamma_n, i_n} $. Let $ \gamma_0:=i_0:=0 $. Suppose that $ 
i_n<\omega $ and $ 
\gamma_n<\mathsf{cf}(\lambda) $ are already defined. We know that $ \beta_n 
\in M_{\beta_{n+1}} $ by Fact \ref{fact: reg contains size}. Therefore  we can choose $ i_{n+1} $ 
and $ \gamma_{n+1}>\gamma_n $  to satisfy  $ 
\beta_n \in M'_{n+1} $.  It follows  that $ \mathcal{D}_{\beta_n} \cup\{ M_{\beta_n}, \mathcal{D}_{\beta_n} \}\subseteq  M'_{n+1}$  by 
property 
\ref{item: superdecom sing} and Fact \ref{fact: submodels subset} where we use that $ \left|\mathcal{D}_{\beta_n}\right| = 
\mathsf{cf}(\lambda)\leq \lambda_{\gamma_{n+1}}= \left|M'_{n+1}\right|$.  In 
particular 
$M'_n \in M'_{n+1} $. Then  $ M'_n \subseteq  M'_{n+1}$ by Fact \ref{fact: submodels subset} because $ 
\left|M'_n\right|=\lambda_{\gamma_n}<\lambda_{\gamma_{n+1}}=\left|M'_{n+1}\right| $. Note that $ \left\langle M'_n:\ n<\omega  \right\rangle 
$ is an increasing (and  continuous) sequence of intersectable and subtractable elementary submodels, therefore $M':= 
\bigcup_{n<\omega} 
M'_n $ is intersectable by Corollary \ref{cor: chain basic}.  Clearly $ M'\subseteq M_\alpha $ because $ M'_n \subseteq M_{\beta_n}\subseteq 
M_\alpha 
$. Suppose first that $ \mathsf{cf}(\lambda)=\omega $. Then $ \sup_{n<\omega}\gamma_n=\omega $ because $ \left\langle \gamma_n:\ 
n<\omega 
\right\rangle $ is strictly increasing. Thus $ \left|M'\right|= \sum_{n<\omega} \left|M'_n\right|=\sum_{n<\omega} 
\lambda_{\gamma_n}=\lambda$. Since $ M_{\beta_n}\in M'_{n+1}\subseteq M' $ and $ \left|M_{\beta_n}\right| =\lambda$, we have $ 
M_{\beta_n}\subseteq M' $ by Fact 
\ref{fact: submodels subset}. It follows that  $ M_\alpha \subseteq M' $ and hence $ M'=M_\alpha $, thus $ M_\alpha $ is intersectable because so 
is $ M' 
$.

Suppose now that $ 
\mathsf{cf}(\lambda) >\omega$. Roughly speaking, we take the whole column from each decomposition instead of a single element, put together 
these columns, and the resulting matrix will be a decomposition of $ M_\alpha $ after the deletion of some initial rows. To make this precise, let 
$ \delta_0:= 
\sup_{n<\omega}\gamma_n<\mathsf{cf}(\lambda) $. We show that for the submodels 
$ M'_{\gamma,n}:=M^{\beta_n}_{\delta_0+\gamma, i_n} $, the family $ \mathcal{D}:=\{ M'_{\gamma,n}:\ \gamma<\mathsf{cf}(\lambda), 
n<\omega \} $ is a decomposition of $ M_\alpha $ and hence $ M_\alpha $ is intersectable by Proposition 
\ref{prop: decomp is intersectable}.  The sequence $ \left\langle \lambda_{\delta_0+\gamma}:\ \gamma<\mathsf{cf}(\lambda)  \right\rangle $ is 
increasing and 
continuous and has limit $ \lambda $ because it is a non-empty terminal segment of such a sequence. We have $\lambda_{\delta_0+\gamma}= 
\left|M'_{\gamma,n}\right| \subseteq M'_{\gamma,n}$ because $ M'_{\gamma,n}=M^{\beta_n}_{\delta_0+\gamma, 
i_n}$ and $ M^{\beta_n}_{\delta_0+\gamma, i_n} $ has these properties by assumption (see  properties (\ref{item: sing size}) and 
(\ref{item: sing contains size}) of 
Definition \ref{def: decomp}). We turn to the verification of property (\ref{item: sing contains previous}). First of all, containment of previous 
elements of the 
same column  in matrix $ \mathcal{D} $ follows directly by
$M'_{\gamma',n}=M^{\beta_n}_{\delta_0+\gamma', i_n} \in M^{\beta_n}_{\delta_0+\gamma, i_n}=M'_{\gamma,n} $ for $ \gamma'<\gamma $ 
(see 
property (\ref{item: sing contains previous})).  Note that the $ n $th column of $ \mathcal{D} $ is a terminal segment of the $ i_n $th column of 
$ \mathcal{D}_{\beta_n} $. We have already seen  that $\mathcal{D}_{\beta_n} \subseteq  M'_{n+1}$, thus
\[ \mathcal{D}_{\beta_n} \subseteq  M'_{n+1}=M^{\beta_{n+1}}_{\gamma_{n+1}, i_{n+1}}\subseteq M^{\beta_{n+1}}_{\delta_0, i_{n+1}}=  
M'_{0,n+1}. \]  
But then for $ m<n+1 $ and $ \gamma<\mathsf{cf}(\lambda) $ we have
\[  \mathcal{D}_{\beta_m}\subseteq M'_{0,m+1}\subseteq M'_{0, 
n+1}\subseteq M'_{\gamma, n+1}.  \]

It remains to check that $ \bigcup_{\gamma<\mathsf{cf}(\lambda), 
n<\omega}M'_{\gamma,n} =M_\alpha$. The containment $ \bigcup_{\gamma<\mathsf{cf}(\lambda), 
n<\omega}M'_{\gamma,n} \subseteq M_\alpha$ is clear because  $ M'_{\gamma,n} \subseteq M_{\beta_n}\subseteq M_\alpha $ for each $ 
\gamma<\mathsf{cf}(\lambda) $ and $ n<\omega $ by the definition of $ 
M'_{\gamma,n} $. To show 
$ \bigcup_{\gamma<\mathsf{cf}(\lambda), n<\omega}M'_{\gamma,n} \supseteq M_\alpha$,  note that  $ \lambda \subseteq 
 \bigcup_{\gamma<\mathsf{cf}(\lambda), n<\omega}M'_{\gamma,n}$ because $ \lambda_{\delta_0+\gamma}= 
 \left|M'_{\gamma,n}\right| \subseteq M'_{\gamma,n}$ and $\lambda= 
 \bigcup_{\gamma<\mathsf{cf}(\lambda)}\lambda_{\delta_0+\gamma} $. Furthermore,  $M_{\beta_n}\in  
 M'_{0,n+1}\subseteq
 \bigcup_{\gamma<\mathsf{cf}(\lambda), n<\omega}M'_{\gamma,n} $ for every $ n<\omega $. But then by Fact \ref{fact: 
 submodels subset}, $ 
 M_{\beta_n} 
\subseteq  
 \bigcup_{\gamma<\mathsf{cf}(\lambda), n<\omega}M'_{\gamma,n}$ for each $ n<\omega $, thus $ M_\alpha \subseteq  
  \bigcup_{\gamma<\mathsf{cf}(\lambda), n<\omega}M'_{\gamma,n} $. Therefore $ M_\alpha $ is decomposable and hence intersectable by 
  Proposition 
  \ref{prop: decomp is intersectable}. This 
  completes 
to proof in the case $ \mathsf{cf}(\lambda)>\omega $ and therefore the proof of $I \supseteq S^{\lambda^{+}}_{\omega} $ is also complete.
\begin{case}\label{case: lambda regular}
$ \kappa=\lambda^{+} $ with $ \lambda=\mathsf{cf}(\lambda)> \aleph_0$
\end{case}
Now we are going to use the approachability assumption in Theorem \ref{thm: main}. Let $ S\subseteq 
S^{\lambda^{+}}_{\lambda} $ be stationary and approachable (see Definition \ref{def: approach ideal}). To demonstrate that $ I $ is stationary, it
is sufficient to show 
that $I \supseteq S $. Let $ \alpha \in 
S $ be fixed. It is 
enough to prove  that  $ M_\alpha $  is decomposable. By the approachability of $ S $, there is a $ c_\alpha \subseteq \alpha 
$ 
cofinal in $ \alpha $ with $ \mathsf{ot}(c_\alpha)=\lambda $ such 
that all proper initial segments of $ c_\alpha $ are in $ M_{\alpha} $. Let $ \left\langle \beta_\xi:\ \xi<\lambda  
\right\rangle $ be the increasing enumeration of $ c_\alpha $. By modifying  $ c_\alpha $ if necessary, we can 
ensure that 
\begin{itemize}
\item $ \left\langle \beta_\xi:\ \xi<\lambda  
\right\rangle $ is continuous,
\item $ \beta_\xi $ is a successor ordinal for every $ \xi\in 
\lambda\setminus \mathsf{acc}(\lambda) $,
\item $ \left\langle \beta_\zeta:\ \zeta\leq\xi  
\right\rangle \in M_{\beta_{\xi+1}}$ for every $ \xi<\lambda $.
\end{itemize}

  First, we build an 
increasing, 
continuous 
sequence $ \left\langle M'_\xi:\ \xi<\lambda  \right\rangle $ of elementary submodels with $ \bigcup_{\xi<\lambda} 
M'_\xi=M_\alpha$ where $ 
M'_{\xi} $ is subtractable and intersectable for every $ \xi\in 
\lambda\setminus \mathsf{acc}(\lambda) $ and satisfies the properties (\ref{item: reg 
size}), 
(\ref{item: reg contains size})  and (\ref{item: reg contains previous})  of Definition \ref{def: decomp} (with $ \lambda $ in the place of $ \kappa 
$).  After this is done,  Lemma \ref{lem: succ intersectable and subtractable} provides a club of $ \lambda $ such that the corresponding 
subsequence of $ 
\left\langle 
M'_\xi:\ \xi<\lambda  \right\rangle $ is a 
decomposition of $ M_\alpha $.

We build the sequence $ \left\langle M'_\xi:\ \xi<\lambda  \right\rangle $ by transfinite recursion. Let $ M'_0:=M^{\beta_0}_{0} $. For  $\xi\in 
\mathsf{acc}(\lambda) $ we take $ 
M'_\xi:=\bigcup_{\zeta<\xi}M'_\zeta  $. We let $ 
M'_{\xi+1}:=M^{\beta_{\xi+1}}_{\gamma_\xi} $, 
where $ \gamma_\xi $ is the smallest 
ordinal for which $\left\langle \beta_\zeta:\ \zeta\leq\xi  \right\rangle \in M^{\beta_{\xi+1}}_{\gamma_\xi} $ and $ 
\left|M'_{\xi} \right| \subseteq \left|M^{\beta_{\xi+1}}_{\gamma_\xi}\right| $.  Note that $ 
M'_{\xi}\in M'_{\xi+1} $ 
because $ M'_{\xi} $ is definable from the parameters $ \left\langle \beta_\zeta:\ \zeta\leq\xi  \right\rangle $ and $ \left\langle 
(M_\beta,\mathcal{D}_\beta):\ \beta\leq  \beta_{\xi}\right\rangle $ which are both in $ M'_{\xi+1} $ (see property \ref{item: 
superdecom reg}). The 
recursion is done. It follows 
directly from the construction that $ \lambda>\left|M'_{\xi}\right| \subseteq M'_{\xi}$  
for every $ 
\xi<\lambda $ (for limit steps see Fact \ref{fact: reg contains size}), moreover, $ M_{\xi} $ is subtractable and intersectable for  every $ 
\xi\in 
\lambda\setminus \mathsf{acc}(\lambda) $. Hence the proof of the case where $ 
\kappa $ is a successor cardinal is complete. This concludes 
the proof of Proposition \ref{prop: decomp exist}.
\end{proof}
As we have already seen (right after Observation \ref{obs: decomp contains size}), $ \Phi=\Psi $ is implied by the 
conjunction of Propositions 
\ref{prop: decomp is intersectable} and \ref{prop: decomp exist}. The first usage of the approachability assumption  was at 
the induction step of Proposition \ref{prop: decomp exist} from $ \aleph_1 $ to $ \aleph_2 $. 
Therefore it remains to justify in ZFC that every $ G \in \Phi $ of size $  \aleph_2 $ is in $ \Psi $.

\section{Graphs of size $ \aleph_2 $}\label{Sec: approach, aleph2}

Let $ G\in \Phi $ with $ \left|G\right|= \aleph_2 $ be given. First of all, we claim that there is an increasing, continuous 
sequence $ 
\left\langle G_\alpha:\ \alpha<\omega_2  \right\rangle $ of subsets of $ G $ exhausting $ G $ such that $ G_\alpha $ is $ G 
$-subtractable and smaller than $ \aleph_2 $ for each $ \alpha $. Indeed, let $ \left\langle 
M_\alpha:\ \alpha<\omega_2  \right\rangle $ be 
an 
increasing, continuous 
and $ \in $-increasing sequence of elementary submodels containing $ G $ with $\aleph_1= \left|M_\alpha\right| 
\subseteq M_\alpha$ 
for each $ \alpha< \omega_2$.  By Lemma \ref{lem: subtractable club exist} we can assume that $ M_\alpha $ is subtractable for every $ 
\alpha<\omega_2 $ by 
switching to a subsequence corresponding to a suitable club of $ \omega_2 $. Then $ G_\alpha:=G \cap M_\alpha $ is as 
desired.  

Let 
$ \left\langle M'_\alpha:\ \alpha<\omega_2  \right\rangle $ be an increasing, continuous 
and $ \in $-increasing sequence of elementary submodels with   $G  \in M'_0
$ such that 
$\aleph_1= \left|M'_\alpha\right| \subseteq M'_\alpha$  and $ M'_{\alpha} $ is  intersectable for every $ \alpha \in\omega_2\setminus 
\mathsf{acc}(\omega_2) $.  
The existence of  
intersectable  elementary submodels $ M $ with $ \aleph_1=\left|M\right|\subseteq M $ containing a prescribed set was 
shown in ZFC. 
Therefore the
sequence 
$ \left\langle M'_\alpha:\ \alpha<\omega_2  \right\rangle $ can be constructed by a straightforward transfinite 
recursion by taking:  an intersectable elementary submodel $ M_0' $ with $ \aleph_1=\left| M'_{0}\right|\subseteq  
M'_{0}$ containing $ G $, union at limit steps and  an intersectable elementary submodel $ M'_{\alpha+1} $ with $ 
\aleph_1=\left| 
M'_{\alpha+1}\right|\subseteq  
M'_{\alpha+1}$ 
containing $ M'_\alpha $  in successor steps. 

We claim that $ M'_\alpha $ is $ G $-subtractable for each $ \alpha<\omega_2 $.  Let $ \alpha $ be fixed. By elementarity we 
can find a  $ 
\left\langle G_\beta:\ \beta<\omega_2  \right\rangle \in M'_0 $ as described in the first paragraph. Since $ G_\beta $ 
is $ G 
$-subtractable for every $ 
\beta<\omega_2 $, it is enough to 
show that there is a $ \beta<\omega_2 $ with $ G\cap M'_\alpha =G_\beta $. Let $ \beta:= \min\{\gamma<\omega_2: 
(G\setminus G_\gamma)\cap M'_\alpha = \varnothing \}$. Then $ G\cap M'_\alpha \subseteq G_\beta $ is immediate from the 
definition. Let $ \gamma <\beta $ and take an $ e \in  (G\setminus G_\gamma)\cap M'_\alpha $. The smallest ordinal $ 
\delta $ with $ e \in (G_{\delta+1}\setminus G_\delta) $ is at least $ \gamma $ and is definable from $ \left\langle G_\beta:\ 
\beta<\omega_2  \right\rangle  $ and $ e $, thus $ \delta\in M'_\alpha $. But then $ \delta+1 \in M'_\alpha $ and therefore $ 
G_{\delta+1}\in  M'_\alpha  $. It follows by Fact \ref{fact: submodels subset} that $ G_{\delta+1}\subseteq  M'_\alpha  $ 
and hence $ G_{\gamma+1}\subseteq  M'_\alpha  $. Since $ \gamma<\beta $ was arbitrary, this means $ G_\beta \subseteq 
M'_\alpha $ which concludes the proof of $ G\cap M'_\alpha =G_\beta $.

It follows from the $ G $-subtractablity of $ M'_\alpha $ that $ G\setminus M'_\alpha \in \Phi $.  Since $ 
M'_{\alpha+1} $ is intersectable and $ G\setminus M'_\alpha \in M'_{\alpha+1} $, it is in 
particular $ 
G\setminus M'_\alpha $-intersectable. Since $ M'_0 $ is also intersectable and hence $ G $-intersectable, we can apply 
Observation \ref{obs: chain very basic} to 
conclude that $ M':=\bigcup_{\alpha<\omega_2}M'_{\alpha} $ is $ G $-intersectable. But then $ G= G\cap M' \in 
\Psi$.
\end{proof}

\begin{question}
Is it possible to prove $ \Phi=\Psi $ in ZFC in the context of Theorem \ref{thm: main}?
\end{question}

\printbibliography
\end{document}